\documentclass[a4paper,12pt]{article}
\usepackage[cp1251]{inputenc}
\usepackage[english]{babel}
\usepackage[tbtags]{amsmath}
\usepackage{amsfonts,amssymb,mathrsfs,amscd}
\usepackage{color}
\usepackage{graphics}
\usepackage{wrapfig}
\usepackage{euscript}
\usepackage[dvips]{graphicx}

\date{}
\newtheorem{theorem}{Theorem}[section]
\newtheorem{lemma}{Lemma}[section]

\newtheorem{definition}{Definition}[section]

\begin{document}
\renewcommand{\thesection}{\arabic{section}}
\renewcommand{\theequation}{\thesection.\arabic{equation}}
\csname @addtoreset\endcsname{equation}{section}
\large

\begin{center}
\bf Construction of step scaling functions in the Vilenkin group
\end{center}

\centerline{\bf Sergei Lukomskii\footnote{Sergei Lukomskii:\ LukomskiiSF@info.sgu.ru}}

Affiliation: Department of Mathematics and Mechanics, Saratov State University, Saratov, Russian Federation,

\begin{abstract}
In Vilenkin's group we present an algorithm for constructing a step scaling function with a given support and a constant on given cosets.   Bibliography: 20 titles.
\end{abstract}

\section{Introduction}We consider the problem of constructing a step refinable function in the Vilenkin group G. Refinable functions are the main tool in constructing MRA, frames  both in the classical case and in Vilenkin groups. They are also used in the definition of the subdivision operator. In applications, it is important to have step-wise refinable  functions with compact support. G.S. Berdnikov \cite{SLGB} proposed using N-valid trees $(N\in \mathbb N)$ to construct step refinable functions. The $N$-valid tree of height $H$  generates a refinable function constant on cosets $G_{H-2N+1}$ with compact support $K\subset G_{-N}$. (The notion of Vilenkin group and related concepts will be explained in details in Section 2)  However, constructing the required tree is not an obvious task.  Yu. Farkov, M. Skopina \cite{YuFMS}
obtained a necessary and sufficient condition for a step function to be refinable and for $N=2$ obtained a sufficient condition for a step function $m_0$ to be a mask. In this article, for $N\ge M\ge 1$ we propose a method for constructing a step refinable  function $\varphi$, for which
$\hat\varphi\in {\mathfrak D}_{G_{-N}^\bot}(G_M^\perp)$.

\section{Preliminaries}\label{s2}
We will consider the Vilenkin group as a  locally compact
zero-dimensional abelian group
 with additional condition $pg_n=0$. Therefore we start with some basic notions
  and facts related to analysis on zero-dimensional groups. One may find more information
   on the topic in \cite{AVDR},\cite{LS1},\cite{LS2}.

Let $(G,\dot + )$~be a~locally compact zero-dimensional Abelian
group with the topology generated by a~countable system of open
subgroups
$$
\cdots\supset G_{-n}\supset\cdots\supset G_{-1}\supset G_0\supset
G_1\supset\cdots\supset G_n\supset\cdots
$$
where
$$
\bigcup_{n=-\infty}^{+\infty}G_n= G,\quad
 \quad \bigcap_{n=-\infty}^{+\infty}G_n=\{0\},
$$
 $p$ is an order of quotient groups $G_n/G_{n+1}$ for all $n\in\mathbb Z$.
 We will always
assume that $p$ is a prime number. We will name such chain
as \it basic chain. \rm  In this case, a~base of the topology is
formed by all possible cosets~$G_n\dot + g$, $g\in G$.

Let~$\mu$ be a Haar measure on~$G$, we know that $\mu
G_n=\frac{1}{p^n}$. Further, let
$\smash[b]{\displaystyle\int_{G}f(x)\,d\mu(x)}$~be the absolutely
convergent integral of the measure~$\mu$.

Given $n\in\mathbb Z$, consider an element $g_n\in G_n\setminus
G_{n+1}$ and fix~it. Then any $x\in G$ has a~unique representation
in the form
\begin{equation}
\label{eq1.1} x=\sum_{n=-\infty}^{+\infty}a_ng_n, \qquad
a_n=\overline{0,p-1}.
\end{equation}
The sum~\eqref{eq1.1} contains finite number of terms with negative
subscripts, that~is,
\begin{equation}
\label{eq1.2} x=\sum_{n=m}^{+\infty}a_ng_n, \qquad
a_n=\overline{0,p-1}, \quad a_m\ne 0.
\end{equation}
We will name system $(g_n)_{n\in \mathbb Z}$ as {\it a basic
system}. The mapping $\lambda:G\to[0,+\infty)$ defined by the equality
$$\lambda(x)=\sum_{n=m}^{+\infty}a_np^{-n-1}$$
is called Monna mapping \cite{Mon}.

Classical examples of zero-dimensional groups are Vilenkin groups
and groups of $p$-adic numbers (see~\cite[Ch.~1, \S\,2]{AVDR}).
 The elements of the~Vilenkin group are infinite sequences
 $x=(x_k)_{k=-\infty}^{+\infty}$ such that:
 \begin{itemize}
 \item[1)] $x_k=\overline{0,p-1}$; \item[2)] only a~finite number
 of~$x_k$ with negative subscripts are different from zero;
 \item[3)] the group operation~$\dot+ $ is the coordinate-wise
 addition modulo~$p$, that~is,
 $$
 x\dot+ y=(x_k\dot+ y_k), \qquad x_k\dot+ y_k=(x_k+y_k)\ \
 \operatorname{mod}p.
 $$
 \end{itemize}
 A topology on such group is generated by the chain of subgroups
 $$
  G_n=\bigl\{x\in G:x=(\dots,0,0,\dots,0,x_n,x_{n+1},\dots),\
   x_\nu=\overline{0,p-1},\ \nu\ge n\bigr\}.
 $$
 It is evident that $\lambda(G_n)=[0,p^{-n}]$.
 The elements $g_n=(\dots,0,0,1,0,0,\dots)$ form a basic system.
 From definition of the operation $\dot+$ we have $pg_n=0$.
 Therefore we will name a zero-dimensional group $(G,\dot+)$ with the condition $pg_n=0$ as Vilenkin group.

By $X$ we denote  the collection
 of the characters of a~group $(G,\dot+ )$; it is
a~group with respect to multiplication, too. Also let
$G_n^\bot=\{\chi\in X:\forall\,x\in G_n\  , \chi(x)=1\}$ be the
annihilator of the group~$G_n$. Each annihilator~$ G_n^\bot$ is
a~group with respect to multiplication, and the subgroups~$
G_n^\bot$ form an~increa\-sing sequence
\begin{equation}
\label{eq1.3} \cdots\subset G_{-n}^\bot\subset\cdots\subset
G_0^\bot \subset G_1^\bot\subset\cdots\subset
G_n^\bot\subset\cdots
\end{equation}
with
$$
\bigcup_{n=-\infty}^{+\infty} G_n^\bot=X \quad\text{and} \quad
\bigcap_{n=-\infty}^{+\infty} G_n^\bot=\{1\},
$$
the quotient group $ G_{n+1}^\bot/ G_n^\bot$ having order~$p$.
The group of characters~$X$ is a zero-dimensional group with a
basic chain \eqref{eq1.3}. The group  may be supplied with the
topology using the chain of subgroups~\eqref{eq1.3}, the family of
the cosets $ G_n^\bot\cdot\chi$, $\chi\in X$, being taken
as~a~base of the topology. The collection of such cosets, along
with the empty set, forms the~semiring~${\mathscr X}$. Given
a~coset $ G_n^\bot\cdot\chi$, we define a~measure~$\nu$ on it by
$\nu( G_n^\bot\cdot\chi)=\nu( G_n^\bot)= p^n$. The measure~$\nu$ can be
extended onto the $\sigma$-algebra of measurable sets in the
standard way. One then forms the absolutely convergent integral
$\displaystyle\int_XF(\chi)\,d\nu(\chi)$ using this measure.

The value~$\chi(g)$ of the character~$\chi$ at an element $g\in G$
will be denoted by~$(\chi,g)$. The Fourier transform~$\widehat f$
of an~$f\in L_2( G)$  is defined~as follows
$$
\widehat f(\chi)=\int_{ G}f(x)\overline{(\chi,x)}\,d\mu(x)=
\lim_{n\to+\infty}\int_{ G_{-n}}f(x)\overline{(\chi,x)}\,d\mu(x),
$$
with the limit being in the norm of $L_2(X)$. For any~$f\in L_2(G)$,
the inversion formula is valid
$$
f(x)=\int_X\widehat f(\chi)(\chi,x)\,d\nu(\chi)
=\lim_{n\to+\infty}\int_{ G_n^\bot}\widehat
f(\chi)(\chi,x)\,d\nu(\chi);
$$
here the limit also signifies the convergence in the norm of~$L_2(
G)$. If $f,g\in L_2( G)$ then the Plancherel formula is valid
$$
\int_{ G}f(x)\overline{g(x)}\,d\mu(x)= \int_X\widehat
f(\chi)\overline{\widehat g(\chi)}\,d\nu(\chi).
$$
\goodbreak

Provided with this topology, the group of characters~$X$ is
a~zero-dimensional locally compact group; there is, however,
a~dual situation: every element $x\in G$ is a~character of the
group~$X$, and~$ G_n$ is the annihilator of the group~$ G_n^\bot$. In particular, Vilenkin group is isomorphic to its group of characters. Therefore, in some cases it is convenient to consider the properties of characters group as dual to the properties of Vilenkin group.

The union of disjoint sets $E_j$ we will denote by $\bigsqcup
E_j$.

 For any $n\in \mathbb Z$ we choose a character $r_n\in  G_{n+1}^{\bot}\backslash G_n^{\bot}$
 and fixed it. The collection of functions $(r_n)_{n\in \mathbb Z}$ is called a Rademacher system. Any character $\chi$ can be rewritten as a product
 $$\chi=\prod_{j=-m}^{+\infty} r_j^{\alpha_j},\ \alpha_j=\overline{0,p-1}.$$
 Then the Fine mapping $\lambda:X\to[0,+\infty)$ can be defined as
 $$\lambda(\chi)=\sum_{j=-\infty}^{m}\alpha_{j}p^{j}.$$
 It is evident that $\lambda(G^\bot_n)=[0,p^n]$.
  Let us denote
  $$
    H_0=\{h\in G: h=a_{-1}g_{-1}\dot+a_{-2}g_{-2}\dot+\dots \dot+ a_{-s}g_{-s}, s\in \mathbb
    N,\ a_j=\overline{0,p-1}\},
  $$
  $$
    H_0^{(s)}=\{h\in G: h=a_{-1}g_{-1}\dot+a_{-2}g_{-2}\dot+\dots \dot+
    a_{-s}g_{-s},\ a_j=\overline{0,p-1}
    \},s\in \mathbb N.
  $$
  Under the Fine mapping $\lambda(H_0)= \mathbb N_0=\mathbb N\bigsqcup \{0\}$ and $\lambda(H_0^{(s)})=\mathbb N_0\bigcap [0,p^{s-1}].$ Thus the set $H_0$ is an analog of the set $\mathbb N_0$.

 We can define the mapping ${\cal
 A}\colon G\to G$ by
 ${\cal A}x:=\sum_{n=-\infty}^{+\infty}a_ng_{n-1}$, where
 $x=\sum_{n=-\infty}^{+\infty}a_ng_n\in G$.  The mapping~${\cal A}$ is called
 a  dilation operator if~${\cal A}(x\dot+ y)={\cal A}x\dot + {\cal A}y$ for all
 $x,y\in G$. By definition, put $(\chi {\cal A},x)=(\chi, {\cal
 A}x)$.
It is also clear that
 ${\cal A} g_n= g_{n-1}, r_n{\cal A} = r_{n+1}$,\
 ${\cal A} G_n= G_{n-1}, G_n^\bot{\cal A}= G_{n+1}^\bot$.

  \begin{lemma}[\cite{LS2}]
 For any zero-dimensional group\\
 1) $\int\limits_{G_0^\bot}(\chi,x)\,d\nu(\chi)={\bf 1}_{G_0}(x)$,
 2) $\int\limits_{G_0}(\chi,x)\,d\mu(x)={\bf 1}_{G_0^\bot}(\chi)$.\\
 \end{lemma}

 \begin{lemma}[\cite{LS2}]
  If the mapping ${\cal A}$ is additive then \\
  1) $\int\limits_{G_n^\bot}(\chi,x)\,d\nu(\chi)=p^n{\bf
  1}_{G_n}(x)$,\\
  2) $\int\limits_{G_n}(\chi,x)\,d\mu(x)=\frac{1}{p^n}{\bf
  1}_{G_n^\bot}(\chi)$.
  \end{lemma}
\begin{lemma}[\cite{LS2}]
Let  $\chi_{n,s}=r_n^{\alpha_n}r_{n+1}^{\alpha_{n+1}}\dots
r_{n+s}^{\alpha_{n+s}}$ be a character which does not belong to
$G_n^\bot$. Then
$$
\int\limits_{G_n^\bot\chi_{n,s}}(\chi,x)\,d\nu(\chi)=p^n(\chi_{n,s},x){\bf
1}_{G_n}(x).
$$
\end{lemma}
\begin{lemma}[\cite{LS2}]
Let
$h_{n,s}=a_{n-1}g_{n-1}\dot+a_{n-2}g_{n-2}\dot+\dots\dot+a_{n-s}g_{n-s}\notin
G_n$. Then
$$
\int\limits_{G_n\dot+h_{n,s}}(\chi,x)\,d\mu(x)=\frac{1}{p^n}(\chi,h_{n,s}){\bf
1}_{G_n^\bot}(\chi).
$$
\end{lemma}
 \begin{definition}
Let $M,N\in\mathbb N$.
We denote by  ${\mathfrak D}_{G_M}(G_{-N})$ the set of functions
 $f\in L_2(G)$ such that 1) ${\rm supp}\,f\subset G_{-N}$, and 2)
 $f$ is constant on cosets $G_M\dot+g$. The class ${\mathfrak
 D}_{G_{-N}^\bot} (G_{M}^\bot)$ is defined similarly.
 \end{definition}

\section{Construction of a step refinable function in the Wilenkin group. Case $M\ge 1$}
Let $M,N\in \mathbb N,N\ge M \ge 1 $.
 We want to construct a refinable function
 $\varphi\in \mathfrak D_{G_M}(G_{-N})$, i.e.
$\hat{\varphi}\in \mathfrak D_{G_{-N}^\bot}(G_{M}^\bot)$.
Let us write the refinable equation in the Fourier domain

 \begin{equation}\label{eq2.1}
 \hat\varphi{(\chi)}=\hat\varphi(\chi A^{-1})m_0(\chi),
  \end{equation}
 with the mask
  \begin{equation} \label{eq2.2}
  m_0(\chi)=\sum_{h\in
  H_0^{(N+1)}}\beta_h\overline{(\chi,A^{-1}h)}
 \end{equation}
 It is known \cite{LS2} that :\\
 1)$m_0$  is constant on cosets  $G_{-N}^\bot r_{-N}^{\alpha_{-N}}...r_{-N+s}^{\alpha_{-N+s}}$,\\
  2)the mask  $m_0(\chi)$ is a periodic  function with any period
 $r_1^{\alpha_1}r_2^{\alpha_2}\dots r_s^{\alpha_s}$ $(s\in\mathbb
 N,\; \alpha_j=\overline{0,p-1},\;j=\overline{1,s})$,\\
3)the mask $m_0(\chi)$ is defined by its values on cosets $G_{-N}^\bot r_{-N}^{\alpha_{-N}}\dots r_0^{\alpha_0}$
$(\alpha_j=\overline{0,p-1})$.


 Let $m_0(G_{-N}^\bot )=1$. If $\hat{\varphi}(\chi)=0$ outside the set $G_{M}^\bot$, then
$$
\hat{\varphi}(\chi)=m_0(\chi)m_0(\chi\mathcal A^{-1})...m_0(\chi\mathcal A^{-N-M}).
$$
Denote by
 \begin{equation} \label{eq2.3}
 m_0(G_{-N}^\bot r_{-N}^{\alpha_{-N}}r_{-N+1}^{\alpha_{-N+1}}...r_{0}^{\alpha_{0}}...r_{M}^{\alpha_{M}})=: \lambda_{ \alpha_{-N}\alpha_{-N+1}...\alpha_{0}...\alpha_{M}}=\lambda_n,
 \end{equation}
 the values of the mask on  $G_{M+1}^\bot$,
 where
  \begin{equation} \label{eq2.4}
 n= \alpha_{-N}+\alpha_{-N+1}p+...+\alpha_{0}p^N +...+\alpha_M p^{N+M} ,
 \end{equation}
  We need to find $\lambda_n$  so that
$$
\hat{\varphi}(\chi)=m_0(\chi)m_0(\chi\mathcal A^{-1})...m_0(\chi\mathcal A^{-N-M})=0
$$
on the set $G_{M+1}^\bot\setminus G_{M}^\bot$ and $\hat{\varphi}(\chi)\neq 0$ for some $\chi \in G_M^\bot \setminus G_{M-1}^\bot$.
To find $\lambda_n$  we construct a rooted  mask tree $T=T(m_0)$ in the following way.
For some $n\in \mathbb N: p^{M+N}\le n\le p^{M+N+1}-1$ we construct the path
$$
\lambda_n\rightarrow \,\lambda_{n \,{\rm div}\, p} \rightarrow \lambda_{n \,{\rm div}\, p^2}\dots \rightarrow \lambda_{n \,{\rm div}\, p^{M+N}}\rightarrow \lambda_0=1
$$
 from the leaf $\lambda_n$ to the root $\lambda_0$ of the  tree $T(m_0)$, where $a \,{\rm div}\, b$ stands for the integer division of $a$ by $b$. It is clear that $H=M+N+1$ is a high of this  tree $T$. (See Figure 1 for the tree $T$).

\unitlength=0.80mm
  \begin{picture}(240,100)
 \small
\put(111,71){\line (-1,1){16}}
     \put(112,71){\line (0,1){16}}
     \put(113,71){\line (1,1){16}}
     \put(90,88){$\lambda_{n\dot+1}$}
      \put(125,88){$\lambda_{n\dot+(p-1)}$}
      \put(108,88){$\lambda_{n}$}

  \put(150,48){\line (-1,1){16}}
     \put(151,48){\line (0,1){16}}
     \put(152,48){\line (1,1){16}}
     \put(125,66){$\lambda_{p^{N+M}-p}$}
      \put(160,66){$\lambda_{p^{N+M}-1}$}

     \put(58,48){\line (-1,1){16}}
    \put(59,48){\line (0,1){16}}
     \put(60,48){\line (1,1){16}}
      \put(39,70){\line (0,1){16}}
      \put(169,70){\line (0,1){16}}
       \put(38,87){$\lambda_{p^{N+M}}$}
\put(160,87){$\lambda_{p^{N+M+1}-1}$}
     \put(38,66){$\lambda_{p^{N+M-1}}$}
      \put(70,66){$\lambda_{p^{N+M-1}+p-1}$}
       \put(105,66){$\lambda_{n\  {\rm div}\  p}$}
      \multiput(51,64)(3,0){6}{$\cdot$}

    \multiput(60,46)(3,0){31}{$\cdot$}
   \put(130,25){\line (-1,1){16}}
     \put(131,25){\line (0,1){16}}
     \put(132,25){\line (1,1){16}}
     \put(110,43){$\lambda_{p^{N+1}-p}$}
      \put(146,43){$\lambda_{p^{N+1}-1}$}
      \multiput(124,41)(3,0){6}{$\cdot$}
      \multiput(74,41)(3,0){4}{$\cdot$}

  \put(124,19){$\lambda_{p-1}$}
  \put(81,19){$\lambda_{1}$}
   \put(105,19){$\lambda_{\nu}$}
   \put(80,25){\line (-1,1){16}}
     \put(81,25){\line (0,1){16}}
     \put(82,25){\line (1,1){16}}
     \put(60,43){$\lambda_{p^{N}}$}
      \put(86,43){$\lambda_{p^{N}+p-1}$}

  \multiput(81,22)(3,0){17}{$\cdot$}

 \put(100,-2){$\lambda_{0}=1$}
  \put(100,2){\line (-1,1){16}}
  \put(106,2){\line (0,1){16}}
  \put(110,2){\line (1,1){16}}
  \put(-5,66){$ G_{M}^\bot \setminus G_{M-1}^\bot $:}
  \put(-5,86){$ G_{M+1}^\bot \setminus G_{M}^\bot $:}
  \put(-5,43){$ G_{1}^\bot \setminus G_{-0}^\bot $:}
  \put(-5,19){$ G_{-N+1}^\bot \setminus G_{-N}^\bot $:}
  \put(-5,-2){$ G_{-N}^\bot  $:}

    \end{picture}\\

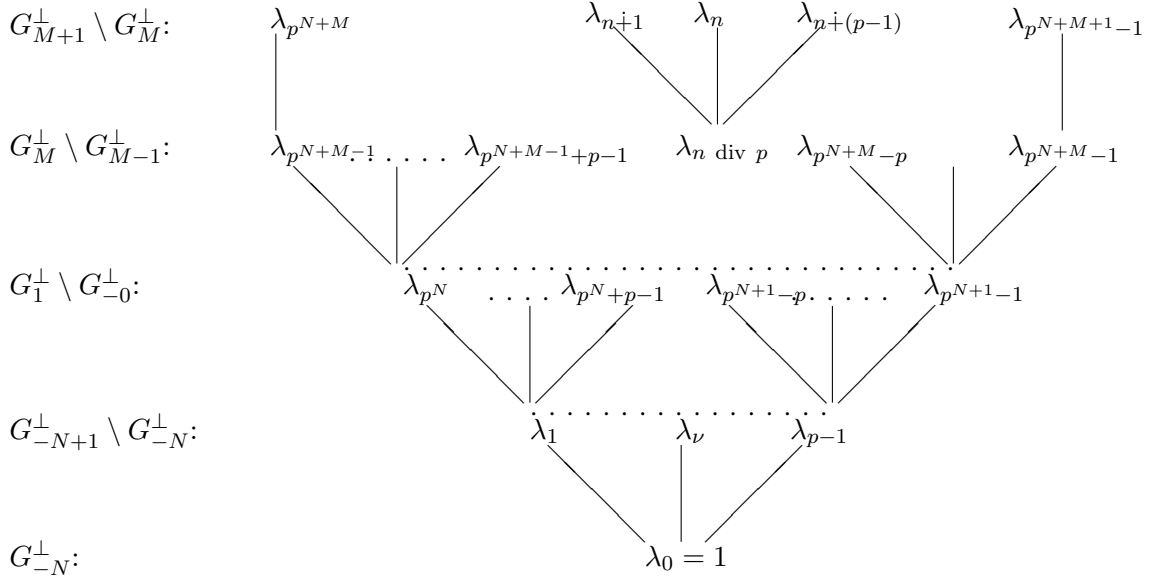
\hskip4cm Figure 1.  The graph of the tree $T=T(m_0)$.\\

 To construct a scaling function $\varphi$ for which $\hat\varphi\in \mathfrak{D}_{G_{-N}^\bot}(G_M^\bot)$, the nodes $\lambda_j$ must be chosen so that:\\
 A1) on each branch with a leaf in $ G_{M+1}^\bot \setminus G_{M}^\bot $ there is at least one zero,\\
A2) there exists at least one path from the root $\lambda_0$ to a leaf (node) $\lambda_{p^{N+M-1}+l}$ at the level $G_{M}^\bot \setminus G_{M-1}^\bot $ on which all $\lambda_\nu\ne 0$.

\begin{theorem}
We select the value $\lambda_n$ at level $G^\bot_{M+1}\setminus G^\bot_{M}$ so that \\
1) in the  $p$-ary representation of the number
$$
n=\alpha_{-N}p^0+\alpha_{-N+1}p^1+...+\alpha_{0}p^N+\alpha_{1}p^{N+1}+
...+\alpha_{M-1}p^{N+M-1}+\alpha_{M}p^{N+M}
$$
all coefficients  $\alpha_k>0$,\\
2) vectors  \\
${\bf a}_{-N+1}=(\alpha_{-N+1},\alpha_{-N+2},...,\alpha_{-1},\alpha_{0})$\\
${\bf a}_{-N+2}=(\alpha_{-N+2},\alpha_{-N+3},...,\alpha_{0},\alpha_{1})$\\
.................................................\\
${\bf a}_{-N+M+1}=(\alpha_{-N+M+1},\alpha_{-N+M+2},...,\alpha_{M-1},\alpha_{M})$\\
\noindent

of length $N$ are pairwise distinct.\\
 Then for any  $\alpha_{-N}=1,2,...,p-1$ and  $\alpha_{M}=1,2,...,p-1$ the path
 \begin{equation} \label{eq3.159}
 \lambda_n\rightarrow\lambda_{n\  {\rm div} \ p}\rightarrow\lambda_{({n\  {\rm div}^2 \ p})} \rightarrow ...\rightarrow \lambda_{({n\  {\rm div}^{N+M-1} \ p})}\rightarrow \lambda_{({n\  {\rm div}^{N+M} \ p})}
 \end{equation}
  defines  a refinable  function  $\varphi$, for which $\hat{\varphi}\in \mathfrak{D}_{G_{-N}^\bot}(G_M^\bot)$.
\end{theorem}
{\bf Proof.}  (See Figure 2) We assume\\
1)  $\lambda_n=\lambda_{n\dot+j}$=0 for $j=0,1,..,p-1$,\\
2) $\lambda_{n\ {\rm div}\ p}=1$ but  $\lambda_{(n\ {\rm div} \ p)\dot+j}=0$ for $j=1,..,p-1$,\\
3) $\lambda_{n\ {\rm div}^\nu p}=1$, $\lambda_{(n\  {\rm div}\ p^\nu)\dot+j}$=0 for $j=1,..,p-1$, $\nu=1,2,..,M+N$ . \\
   This is possible, since according to the condition of the theorem  the vectors  ${\bf a}_{-N+1},...,{\bf a}_{-N+M+1}$ are pairwise distinct. Convinced of it.
  The principle for choosing values $\lambda_{n\ {\rm div}^\nu p}$ is illustrated in Figure 2.

\unitlength=0.80mm
  \begin{picture}(240,56)
 \small
  \put(124,49){$\lambda_{n\dot+(p-1)}=0$}
  \put(81,49){$\lambda_{n\dot+1}=0$}
   \put(105,49){$\lambda_{n}=0$}

 \put(98,28){$\lambda_{n\ {\rm div}\ p}=1$}
  \put(100,32){\line (-1,1){16}}
    \put(106,32){\line (0,1){16}}
  \put(110,32){\line (1,1){16}}
   \put(0,48){$G_{M+1}^\bot \setminus G_{M}^\bot $:}

 \put(0,28){$ G_{M}^\bot \setminus G_{M-1}^\bot $:}
   \put(0,-4){$ G_{M-\nu +1}^\bot \setminus  G_{M-\nu}^\bot $:}
   \put(124,29){$\lambda_{n\ {\rm div}\ p\dot+k}=0$}
  \put(68,29){$\lambda_{n\ {\rm div}\ p\dot+l}=0$}

  \put(100,10){\line (-1,1){16}}
  \put(106,10){\line (0,1){16}}
  \put(110,10){\line (1,1){16}}

  \multiput(0,2)(10,0){20}{$\cdots$}
  \put(96,-4){$\lambda_{n\ \rm div \ p^\nu}=1$}
\put(126,-4){$\lambda_{n\ \rm div \ p^\nu \dot+k}=0$}
\put(60,-4){$\lambda_{n\ \rm div \ p^\nu \dot+l}=0$}
   \put(60,8){$\lambda_{n\ \rm div \ p^2 \dot+l}=0$}
    \put(96,8){$\lambda_{n\ \rm div \ p^2}=1$}
    \put(126,8){$\lambda_{n\ \rm div \ p^2 \dot+k}=0$}

    \end{picture}\\

\vskip1cm
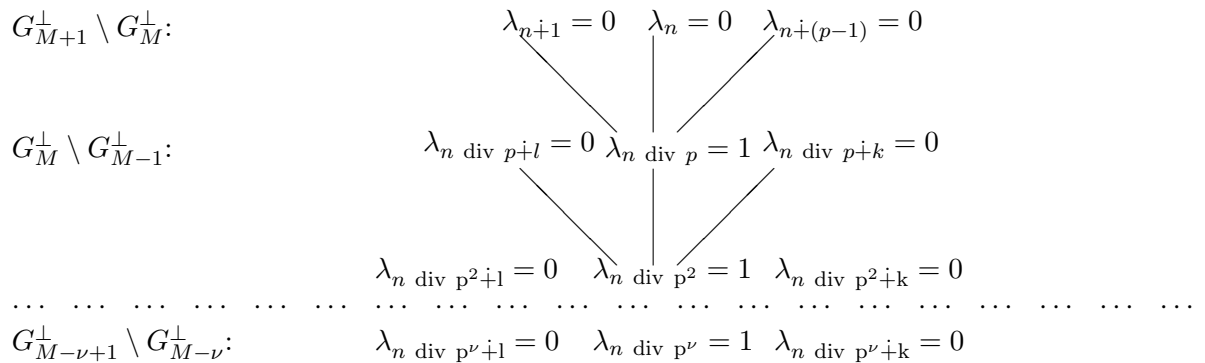
\hskip4cm Figure 2. Subtree with  the leaf  $\lambda_n$ .\\

Let's consider several cases.\\
 Case 1. Consider a node $\lambda_n$.\\
 From the p-ary representation of the number $n$, we find

$$\lambda_n=m_0(G_{-N}^\bot r_{-N}^{\alpha_{-N}}r_{-N+1}^{\alpha_{-N+1}}...r_{-1}^{\alpha_{-1}}r_{0}^{\alpha_{0}}
r_{1}^{\alpha_{1}}...r_{M-2}^{\alpha_{M-2}}r_{M-1}^{\alpha_{M-1}}r_{M}^{\alpha_{M}}).$$

 For convenience, we denote $\tilde\alpha_{j}\neq \alpha_{j}$ and
$$\tilde\lambda_n=m_0(G_{-N}^\bot r_{-N}^{\tilde\alpha_{-N}}r_{-N+1}^{\alpha_{-N+1}}...r_{-1}^{\alpha_{-1}}r_{0}^{\alpha_{0}}
r_{1}^{\alpha_{1}}...r_{M-2}^{\alpha_{M-2}}r_{M-1}^{\alpha_{M-1}}r_{M}^{\alpha_{M}}).
$$

 Using the periodicity of the mask, we have
\begin{equation} \label{eq3.16}
\quad \lambda_n
=m_0(G_{-N}^\bot r_{-N}^{\alpha_{-N}}r_{-N+1}^{\alpha_{-N+1}}.
..r_{-1}^{\alpha_{-1}}r_{0}^{\alpha_{0}}),
\tilde\lambda_n= m_0(G_{-N}^\bot r_{-N}^{\tilde\alpha_{-N}}r_{-N+1}^{\alpha_{-N+1}}...r_{-1}^{\alpha_{-1}}r_{0}^{\alpha_{0}}).
\end{equation}
We require  that $\lambda_n=\tilde\lambda_n= 0$.\\
Case 2. Consider a node  $\lambda_{n\ {\rm  div}\ p}$. Since
$$n\ {\rm  div}\ p=\alpha_{-N+1}p^0+\alpha_{-N+2}p^1+...+\alpha_{1}p^N+
...+\alpha_{M-1}p^{N+M-2}+\alpha_{M}p^{N+M-1},
$$
we have by the periodicity
$$
\lambda_{n\ {\rm  div}\ p}=
m_0(G_{-N}^\bot r_{-N}^{\alpha_{-N+1}}r_{-N+1}^{\alpha_{-N+2}}...r_{-1}^{\alpha_{0}}r_{0}^{\alpha_{1}}),
$$

$$
\tilde\lambda_{n\ {\rm  div}\ p}
=m_0(G_{-N}^\bot r_{-N}^{\widetilde{\alpha}_{-N+1}}r_{-N+1}^{\alpha_{-N+2}}...r_{-1}^{\alpha_{0}}r_{0}^{\alpha_{1}}).
$$
 Since the vectors
$(\alpha_{-N+1},\alpha_{-N+2},...,\alpha_{-1},\alpha_{0})$ and
$(\alpha_{-N+2},\alpha_{-N+3},...,\alpha_{0},\alpha_{1})$

are different (this is given ),  then vectors
$$
(\alpha_{-N},\alpha_{-N+1},\alpha_{-N+2},...,\alpha_{-1},\alpha_{0}), \ (\tilde{\alpha}_{-N},\alpha_{-N+1},\alpha_{-N+2},...,\alpha_{-1},\alpha_{0})
$$
and
$$
(\alpha_{-N+1},\alpha_{-N+2},\alpha_{-N+3},...,\alpha_{0},\alpha_{1}), \
(\tilde{\alpha}_{-N+1},\alpha_{-N+2},\alpha_{-N+3},...,\alpha_{0},\alpha_{1})
$$ are
also different. (We have added new components.)

Therefore we require that

$$
\lambda_{n\ {\rm  div}\ p}=m_0(G_{-N}^\bot r_{-N}^{\alpha_{-N+1}}r_{-N+1}^{\alpha_{-N+2}}...r_{-1}^{\alpha_{0}}r_{0}^{\alpha_{1}})=1,
$$
and
$$
\tilde\lambda_{n\ {\rm  div}\ p}=m_0(G_{-N}^\bot r_{-N}^{\widetilde{\alpha}_{-N+1}}r_{-N+1}^{\alpha_{-N+2}}...r_{-1}^{\alpha_{0}}r_{0}^{\alpha_{1}})=0
$$
(see Figure 2.).

Case 3. For $1<\nu\le N+M$ we have
$$
{n\ {\rm  div}^\nu\ p}=\alpha_{-N+\nu}p^0+\alpha_{-N+\nu+1}p^1+...+\alpha_{\nu-1}p^{N-1}+\alpha_{\nu}p^{N}+
...+\alpha_{M}p^{N+M-\nu},
$$
and
$$
\lambda_{n\ {\rm  div}^\nu\ p}=m_0(G_{-N}^\bot r_{-N}^{\alpha_{-N+\nu}}r_{-N+1}^{\alpha_{-N+\nu+1}}...r_{-1}^{\alpha_{\nu -1}}r_{0}^{\alpha_{\nu}}r_{1}^{\alpha_{\nu+1}}
...r_{M-\nu}^{\alpha_{M}}).
$$

If $1\le \nu\le M$, then, we have  to the periodicity,
\begin{equation} \label{eq3.21}
\lambda_{n\ {\rm  div}^\nu\ p}=m_0(G_{-N}^\bot r_{-N}^{\alpha_{-N+\nu}}r_{-N+1}^{\alpha_{-N+\nu+1}}...r_{0}^{\alpha_{\nu}}),
\end{equation}
\begin{equation} \label{eq3.211}
\tilde\lambda_{n\ {\rm  div}^\nu\ p}=m_0(G_{-N}^\bot r_{-N}^{\tilde\alpha_{-N+\nu}}r_{-N+1}^{\alpha_{-N+\nu+1}}...r_{0}^{\alpha_{\nu}}).
\end{equation}

We require that the equalities
\begin{equation} \label{eq3.212}
\lambda_{n\ {\rm  div}^\nu\ p}=m_0(G_{-N}^\bot r_{-N}^{\alpha_{-N+\nu}}r_{-N+1}^{\alpha_{-N+\nu+1}}...r_{0}^{\alpha_{\nu}})=1,
\end{equation}
\begin{equation} \label{eq3.213}
\tilde\lambda_{n\ {\rm  div}^\nu\ p}=m_0(G_{-N}^\bot r_{-N}^{\tilde\alpha_{-N+\nu}}r_{-N+1}^{\alpha_{-N+\nu+1}}...r_{0}^{\alpha_{\nu}})=0,
\end{equation}
hold for $1\le  \nu\le M$.

This requires that the vectors\\
$(\alpha_{-N+1},\alpha_{-N+2},\alpha_{-N+3},...,\alpha_{0},\alpha_{1}), \quad
(\tilde\alpha_{-N+1},\alpha_{-N+2},\alpha_{-N+3},...,\alpha_{0},\alpha_{1})$\\
$(\alpha_{-N+2},\alpha_{-N+3},\alpha_{-N+4},...,\alpha_{1},\alpha_{2}), \quad
(\tilde\alpha_{-N+2},\alpha_{-N+3},\alpha_{-N+4},...,\alpha_{1},\alpha_{2})$

..............................................................................\\
$(\alpha_{-N+M},\alpha_{-N+M+1},...,\alpha_{M-1},\alpha_{M}),\quad
(\tilde\alpha_{-N+M},\alpha_{-N+M+1},...,\alpha_{M-1},\alpha_{M})$\\
 whose length is $N+1$, be pairwise distinct.\\
These vectors will be pairwise distinct if the vectors
 \\
$(\alpha_{-N+1},\alpha_{-N+2},...,\alpha_{-1},\alpha_{0})$,\\
$(\alpha_{-N+2},\alpha_{-N+3},...,\alpha_{0},\alpha_{1})$,\\

...........................................\\
$(\alpha_{-N+M},\alpha_{-N+M+1},...,\alpha_{M-2},\alpha_{M-1})$,\\
$(\alpha_{-N+M+1},\alpha_{-N+M+2},...,\alpha_{M-1},\alpha_{M})$,\\
 are pairwise distinct. But this is given by the hypothesis.
Note that (\ref{eq3.21}-\ref{eq3.213}) are the values of the mask on
$G_{1}^\bot\setminus G_{0}^\bot  $.

If $M<\nu\le N+M$, then
\begin{equation} \label{eq3.24}
\lambda_{n\ {\rm  div}^\nu\ p}=
m_0(G_{-N}^\bot r_{-N}^{\alpha_{-N+\nu}}r_{-N+1}^{\alpha_{-N+\nu+1}}...r_{M-\nu}^{\alpha_{M}})
\end{equation}
\begin{equation} \label{eq3.25}
\tilde{\lambda}_{n\ {\rm  div}^\nu\ p}=
m_0(G_{-N}^\bot r_{-N}^{\widetilde{\alpha}_{-N+\nu}}r_{-N+1}^{\alpha_{-N+\nu+1}}...r_{M-\nu}^{\alpha_{M}})\quad \widetilde{\alpha}_{-N+\nu}\neq \alpha_{-N+\nu}
\end{equation}
In particular, for $\nu=M+1,M+2, ....,M+N$ , we have ,

\begin{equation} \label{eq3.34}
\lambda_{n\ {\rm  div}^{M+1}\ p}=
m_0(G_{-N}^\bot r_{-N}^{\alpha_{-N+M+1}}r_{-N+1}^{\alpha_{-N+M+2}}...r_{-1}^{\alpha_{M}}r_{0}^0),
\end{equation}

\hskip 2.5cm $ ................................................. ..................    $
\begin{equation} \label{eq3.37}
\lambda_{n\ {\rm  div}^{M+N}\ p}=
m_0(G_{-N}^\bot r_{-N}^{\alpha_{M}}r_{-N+1}^0 r_{-N+2}^0...r_{-1}^0r_{0}^0).
\end{equation}

\begin{equation} \label{eq3.38}
\tilde{\lambda}_{n\ {\rm  div}^{M+1}\ p}=
m_0(G_{-N}^\bot r_{-N}^{\tilde{\alpha}_{-N+M+1}}r_{-N+1}^{\alpha_{-N+M+2}}...r_{-1}^{\alpha_{M}}r_{0}^0),
\end{equation}

\hskip 2.5cm $ ................................................. ..................    $
\begin{equation} \label{eq3.39}
\tilde{\lambda}_{n\ {\rm  div}^{M+N}\ p}=
m_0(G_{-N}^\bot r_{-N}^{\tilde{\alpha}_{M}}r_{-N+1}^0 r_{-N+2}^0...r_{-1}^0r_{0}^0).
\end{equation}

It is clear that
(\ref{eq3.34}-\ref{eq3.39}) are mask $m_0$ values on sets
$$G_{0}^\bot\setminus G_{-1}^\bot,  G_{-1}^\bot\setminus G_{-2}^\bot, ...,G_{-N+2}^\bot\setminus G_{-N+1}^\bot , G_{-N+1}^\bot\setminus G_{-N}^\bot .
$$

Thus, vectors \\
$(\alpha_{-N},\alpha_{-N+1},\alpha_{-N+2},...,\alpha_{-1},\alpha_{0}), \quad
(\tilde\alpha_{-N},\alpha_{-N+1},\alpha_{-N+2},...,\alpha_{-1},\alpha_{0})$\\
$(\alpha_{-N+1},\alpha_{-N+2},\alpha_{-N+3},...,\alpha_{0},\alpha_{1}), \quad
(\tilde\alpha_{-N+1},\alpha_{-N+2},\alpha_{-N+3},...,\alpha_{0},\alpha_{1})$\\
.................................................\\
$(\alpha_{-N+M},\alpha_{-N+M+1},...,\alpha_{M-1},\alpha_{M}), \quad
(\tilde\alpha_{-N+M},\alpha_{-N+M+1},...,\alpha_{M-1},\alpha_{M})$\\

\noindent
$(\alpha_{-N+M+1},\alpha_{-N+M+2},...,\alpha_{M},0)\quad
(\tilde\alpha_{-N+M+1},\alpha_{-N+M+2},...,\alpha_{M},0)$\\
.................................................\\
$(\alpha_{M},0,...,0,0),\quad
(\tilde\alpha_{M},0,...,0,0)$\\
 are pairwise distinct.

 Since the mask is completely determined by its values,  the statements \\
 1)  $\lambda_n=\lambda_{n\dot+j}$=0 for $j=0,1,..,p-1$,\\
2) $\lambda_{n\ {\rm div}\ p}=1$ but  $\lambda_{(n\ {\rm div} \ p)\dot+j}=0$ for $j=1,..,p-1$,\\
3) $\lambda_{n\ {\rm div}^\nu p}=1$, $\lambda_{(n\  {\rm div}\ p^\nu)\dot+j}$=0 for $j=1,..,p-1$, $\nu=1,2,..,M+N$ . \\
 are valid, and the theorem has been proven.\\
In particular $$\lambda_{n\ {\rm  div}^{N+M}p}=m_0(G_{-N}^\bot r_{-N}^{\alpha_{M}})=1,$$
$$\tilde{\lambda}_{n\ {\rm  div}^{N+M}p}=m_0(G_{-N}^\bot r_{-N}^{\widetilde{\alpha}_{M}})=1 \ \ (\widetilde{\alpha}_{M}=0).$$

Thus, for any $\alpha_{-N}=1,2,...,p-1$ , the path  $(\ref{eq3.159})$ generates a scaling function.
$\square$

\section{ Examples }
\subsection{Example 1, $N=M=2, p=3$}
 Select  $$\lambda_n=m_0(G_{-2}^\bot r_{-2}^{\alpha_{-2}} r_{-1}^{\alpha_{-1}}r_{0}^{\alpha_{0}}
r_{1}^{\alpha_{1}}r_{2}^{\alpha_{2}}) $$
 on level $G_3^\bot \setminus G_2^\bot$,
 $n=\alpha_{-2}+p\alpha_{-1}+p^2\alpha_0+p^3\alpha_1+p^4\alpha_2$.
  The numbers $\alpha_j$ are the coefficients of the ternary expansion of the number $n$. We form pairwise distinct vectors (see Theorem 3.1)

  $$ (\alpha_{-1},\alpha_{0})=(2,1),\ (\alpha_{0},\alpha_{1})=(1,2),\ (\alpha_{1},\alpha_{2})=(2,2).
  $$
  Thus, we have obtained the coefficients $\alpha_{-1}=2,\alpha_{0}=1,\alpha_{1}=2,\alpha_{2}=2$.
  For   $\alpha_{-2}=1$ we obtain  $n=1+3\cdot 2+3^2\cdot 1+3^3\cdot 2+3^4\cdot 2 =232$.
    Now we find the path
    $$((((232 \ {\rm div} \ 3) \ {\rm div}\  3) \ {\rm div} \ 3) \ {\rm div}\  3)=232 \rightarrow 77 \rightarrow 25 \rightarrow 8\rightarrow 2$$
      that generates the scaling function  (see Figure 3).

\unitlength=0.60mm
  \begin{picture}(240,120)
 \small

     \put(205,48){\line (-1,1){16}}
     \put(205,48){\line (0,1){16}}
     \put(207,48){\line (1,1){16}}

     \put(182,66){$\lambda_{24}=0$}
     \put(204,66){$\lambda_{25}=1$}
      \put(225,66){$\lambda_{26}=0$}
      \put(202,66){\textcolor{green}{$\bigcirc$}}

      \put(205,69){\line (-1,1){16}}
     \put(206,69){\line (0,1){16}}
     \put(207,69){\line (1,1){16}}

     \put(180,86){$\lambda_{75}=0$}
     \put(201,86){$\lambda_{76}=0$}
      \put(221,86){$\lambda_{77}=1$}
      \put(221,86){\textcolor{green}{$\bigcirc $}}

        \put(223,90){\line (-1,1){16}}
     \put(223,90){\line (0,1){16}}
     \put(223,90){\line (1,1){16}}

      \put(200,108){$\lambda_{231}=$}
     \put(217,108){$\lambda_{232}=$}
      \put(234,108){$\lambda_{233}=0$}
      \put(217,108){\textcolor{red}{$\bigcirc $}}

      \put(229,69){\line (1,1){16}}
      \put(245,86){$\lambda_{80}=0$}

     \put(58,48){\line (-1,1){16}}
     \put(59,48){\line (0,1){16}}
     \put(60,48){\line (-1,2){8}}
     \put(38,66){$\lambda_{9}$}
      \put(55,66){$\lambda_{11}$}
       \put(47,66){$\lambda_{10}$}

        \put(80,48){\line (1,1){16}}
     \put(81,48){\line (0,1){16}}
     \put(82,48){\line (-1,1){16}}

      \put(66,66){$\lambda_{15},$}
      \put(76,66){$\lambda_{16},$}
       \put(84,66){$\lambda_{17}=0$}
       \put(172,72){1}
       \put(160,72){0}
       \put(140,72){0}

              \put(66,72){0}
       \put(76,72){0}
       \put(86,72){0}

   \put(187,25){\line (-3,1){45}}
     \put(187,25){\line (-1,1){16}}
     \put(187,25){\line (1,1){16}}

     \put(130,43){$\lambda_{6}=0$}
     \put(165,43){$\lambda_{7}=0$}
      \put(202,43){$\lambda_{8}=1$}
       \put(202,43){\textcolor{green}{$\bigcirc $}}

       \put(140,66){$\lambda_{21}$}
     \put(160,66){$\lambda_{22}$}
      \put(172,66){$\lambda_{23}$}

      \put(165,45){\line (2,3){12}}
     \put(165,45){\line (0,1){18}}
     \put(165,45){\line (-1,1){18}}


  \put(179,19){$\lambda_{2}=1$}
  \put(179,19){\textcolor{green}{$\bigcirc $}}
  \put(81,19){$\lambda_{1}$}
      \put(80,25){\line (-1,1){16}}
     \put(81,25){\line (0,1){16}}
     \put(82,25){\line (-1,2){8}}
     \put(60,43){$\lambda_{3}$}
      \put(77,43){$\lambda_{5}$}
      \put(69,43){$\lambda_{4}$}

 \put(130,-2){$\lambda_{0}=1$}
  \put(130,0){\line (-2,1){40}}
    \put(150,2){\line (2,1){30}}
  \put(-5,90){$G_2^\bot \setminus G_{1}^\bot $:}
  \put(-5,66){$G_1^\bot \setminus G_{0}^\bot $:}
  \put(-5,43){$G_{0}^\bot \setminus G_{-1}^\bot $:}
  \put(-5,19){$G_{-1}^\bot \setminus G_{-2}^\bot $:}
  \put(-5,-2){$G_{-2}^\bot  $:}

    \end{picture}\\

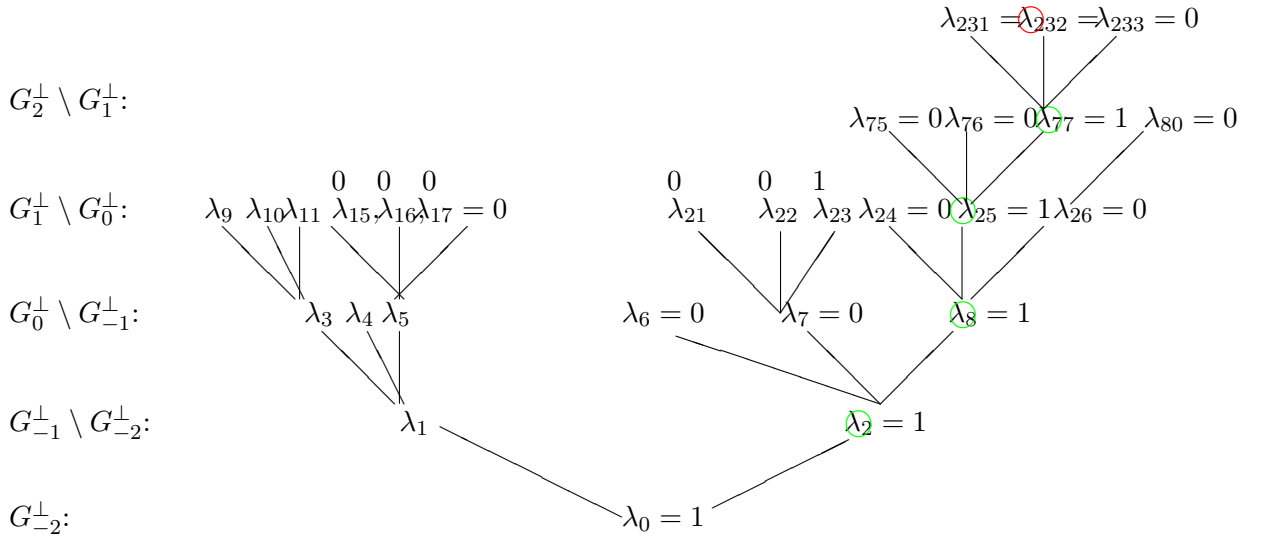
\centerline{Figure 3}

 \noindent
 Using periodicity, we have
        $$ (\lambda_{231}, \lambda_{232},\lambda_{233})=(\lambda_{15}, \lambda_{16},\lambda_{17});
        $$
        $$\quad  (\lambda_{75}, \lambda_{76},\lambda_{77})=(\lambda_{21}, \lambda_{22},\lambda_{23}).
        $$
         Therefore, we choose
         $$(\lambda_{15},\lambda_{16}\,\lambda_{17})=(0,0,0) ;
         (\lambda_{21}, \lambda_{22},\lambda_{23})=(0,0,1), \lambda_1=0.$$
         We choose the values $\lambda_{27},...,\lambda_{74}$ and $\lambda_{81},...,\lambda_{230}$ based on the periodicity conditions.

     \subsection{Example 2. $N=M=2, p=3$}
      In the previous example (see Fig.3) we can take $$  (\lambda_{75}, \lambda_{76},)=(\lambda_{21}, \lambda_{22},)=(1,1);
$$ and
$$
(\lambda_{228}, \lambda_{229},\lambda_{230})=(\lambda_{12}, \lambda_{13},\lambda_{14})=(0,0,0).$$
$$(\lambda_{225}, \lambda_{226},\lambda_{227})=(\lambda_{9}, \lambda_{10},\lambda_{11})=(0,0,0);
$$
  In this case, we obtain 3 paths $$\lambda_0 \rightarrow \lambda_2 \rightarrow \lambda_8 \rightarrow \lambda_{25} \rightarrow \lambda_{77} $$
$$\lambda_0 \rightarrow \lambda_2 \rightarrow \lambda_8 \rightarrow \lambda_{25} \rightarrow \lambda_{76} $$
$$\lambda_0 \rightarrow \lambda_2 \rightarrow \lambda_8 \rightarrow \lambda_{25} \rightarrow \lambda_{75} $$
    that satisfy conditions A1) and A2) (see Fig. 4).

\unitlength=0.60mm
  \begin{picture}(240,120)
   \small
       \put(202,48){\line (-1,1){16}}
     \put(202,48){\line (0,1){16}}
     \put(202,48){\line (1,1){16}}

     \put(185,66){$\lambda_{24}$}
     \put(185,72){$0$}
     \put(198,66){$\lambda_{25}=1$}
      \put(221,66){$\lambda_{26}=0$}
      \put(197,66){\textcolor{green}{$\bigcirc $}}

      \put(200,72){\line (-1,1){12}}
     \put(200,72){\line (0,1){12}}
     \put(200,72){\line (1,1){12}}

     \put(200,92){\line (-1,1){16}}
     \put(200,92){\line (0,1){16}}
     \put(200,92){\line (1,1){16}}

      \put(120,86){$\lambda_{71}$}
      \put(133,86){$\lambda_{18}$}
     \put(142,86){$\lambda_{19}$}
      \put(151,86){$\lambda_{20}$}
     \put(170,86){$1=\lambda_{75}=$}
     \put(198,86){$\lambda_{76}=$}
      \put(213,86){$\lambda_{77}$}
      \put(223,86){$\lambda_{78}$}
            \put(235,86){$\lambda_{79}$}

      \put(213,86){\textcolor{green}{$\bigcirc $}}
      \put(197,86){\textcolor{green}{$\bigcirc $}}
      \put(180,86){\textcolor{green}{$\bigcirc $}}

      \put(176,90){\line (0,1){16}}
      \put(176,90){\line (-1,1){16}}
      \put(176,90){\line (-2,1){32}}

        \put(220,90){\line (1,1){16}}
     \put(220,90){\line (0,1){16}}
     \put(220,90){\line (2,1){32}}

      \put(180,108){$\lambda_{228}$}
     \put(194,108){$\lambda_{229}$}
      \put(208,108){$\lambda_{230}$}

       \put(138,108){$\lambda_{225}$}
     \put(154,108){$\lambda_{226}$}
      \put(168,108){$\lambda_{227}$}

      \put(220,108){$\lambda_{231}$}
     \put(237,108){$\lambda_{232}$}
      \put(254,108){$\lambda_{233}$}

      \put(225,72){\line (0,1){12}}
      \put(225,72){\line (1,1){12}}
      \put(225,72){\line (2,1){24}}

      \put(245,86){$\lambda_{80}=0$}

     \put(43,48){\line (-1,2){8}}
     \put(45,48){\line (0,1){16}}
     \put(46,48){\line (1,2){8}}

     \put(30,66){$\lambda_{9}$}
      \put(48,66){$\lambda_{11}$}
       \put(39,66){$\lambda_{10}$}

       \put(32,72){0}
      \put(51,72){0}
       \put(42,72){0}

       \put(60,66){$\lambda_{12}$}
       \put(71,66){$\lambda_{13}$}
       \put(81,66){$\lambda_{14}$}

        \put(111,48){\line (1,2){8}}
     \put(111,48){\line (0,1){16}}
     \put(111,48){\line (-1,1){16}}

      \put(30,86){$\lambda_{27}$}
      \put(39,86){$\lambda_{28}$}
       \put(48,86){$\lambda_{29}$}

      \put(60,86){$ \ . \ .\  .\ $}
       \put(75,86){$\lambda_{53}$}
       \put(87,86){$\lambda_{54}$}
       \put(96,86){$\lambda_{55 \ .\ .\  .}$}
       \put(105,86){$\lambda_{55 \ .\ .\  .}$}

      \put(93,66){$\lambda_{15},$}
      \put(103,66){$\lambda_{16},$}
       \put(113,66){$\lambda_{17}$}

        \put(124,66){$\lambda_{18}$}
      \put(134,66){$\lambda_{19}$}
       \put(144,66){$\lambda_{20}$}
        \put(124,72){0}
      \put(134,72){0}
       \put(144,72){0}

       \put(94,72){0}
      \put(104,72){0}
       \put(114,72){0}

       \put(175,72){1}
       \put(165,72){1}
       \put(155,72){1}

       \put(62,72){0}
       \put(72,72){0}
       \put(82,72){0}

       \put(168,25){\line (0,1){16}}
     \put(168,25){\line (-1,1){16}}
     \put(168,25){\line (2,1){32}}

     \put(143,45){$\lambda_{6}$}
     \put(168,43){$\lambda_{7}=0$}
      \put(198,43){$\lambda_{8}=1$}
       \put(198,43){\textcolor{green}{$\bigcirc $}}

       \put(155,66){$\lambda_{21}$}
     \put(165,66){$\lambda_{22}$}
      \put(173,66){$\lambda_{23}$}

      \put(168,45){\line (2,3){12}}
     \put(168,45){\line (0,1){18}}
     \put(168,45){\line (-2,3){12}}

\put(143,48){\line (-1,1){15}}
     \put(143,48){\line (-1,3){5}}
     \put(143,48){\line (1,2){8}}

   \put(70,25){\line (-3,2){24}}
     \put(70,25){\line (0,1){16}}
     \put(70,25){\line (2,1){38}}

     \put(43,43){$\lambda_{3}$}
      \put(110,43){$\lambda_{5}$}
      \put(69,43){$\lambda_{4}$}

     \put(165,19){$\lambda_{2}=1$}
     \put(165,19){\textcolor{green}{$\bigcirc $}}
     \put(68,19){$\lambda_{1}$}

        \put(73,48){\line (1,2){8}}
     \put(73,48){\line (0,1){16}}
     \put(73,48){\line (-1,2){8}}

 \put(130,-2){$\lambda_{0}=1$}
  \put(130,0){\line (-3,1){57}}
  \put(149,3){\line (1,1){16}}

  \put(-5,108){$G_3^\bot \setminus G_{2}^\bot $:}
  \put(-5,87){$G_2^\bot \setminus G_{1}^\bot $:}
  \put(-5,66){$G_1^\bot \setminus G_{0}^\bot $:}
  \put(-5,43){$G_{0}^\bot \setminus G_{-1}^\bot $:}
  \put(-5,19){$G_{-1}^\bot \setminus G_{-2}^\bot $:}
  \put(-5,-2){$G_{-2}^\bot  $:}
    \end{picture}\\
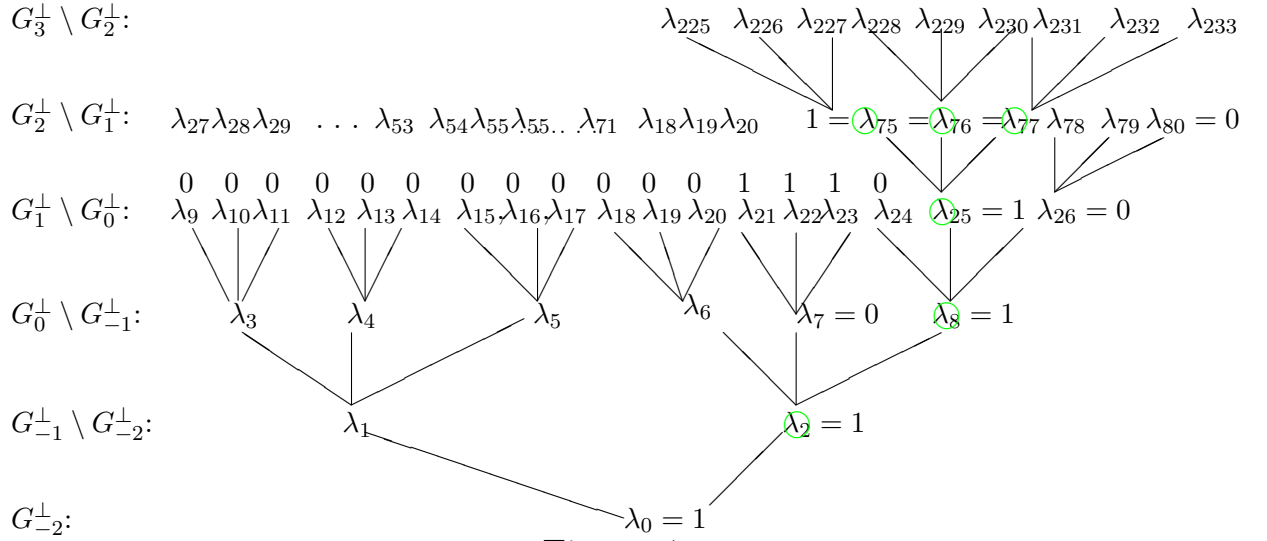
    \centerline{Figure 4}
   \noindent
    Due to periodicity, we have
 $$
(\lambda_{0+27j}, \lambda_{1+27j}, \ .\ .\ ,
\lambda_{26+27j})= (\lambda_{0},\lambda_{1},\ .\ .\ \lambda_{26})\  j=0,1,...,8.
$$
in particular,
$$
(\lambda_{27}, \lambda_{28},\lambda_{29} \ .\ .\ \lambda_{53})= (\lambda_{0},\lambda_{1},\ .\ .\ \lambda_{26})
$$
$$
(\lambda_{54}, \lambda_{55},\lambda_{56} \ .\ .\ \lambda_{71} \ . \ . \ . \lambda_{80})= (\lambda_{0},\lambda_{1},\ .\ .\ \lambda_{26})
$$
$$
(\lambda_{225},\lambda_{226},\lambda_{227})=\lambda_{228}, \lambda_{229},\lambda_{230})=(0,0,0)
$$
$$
(\lambda_{234}, \lambda_{235},\lambda_{236})=(\lambda_{18}, \lambda_{19},\lambda_{20})=(0,0,0)
$$
$$
(\lambda_{231}, \lambda_{232},\lambda_{233})=(\lambda_{15}, \lambda_{16},\lambda_{17})=(0,0,0);$$
$$  (\lambda_{75}, \lambda_{76},\lambda_{77})=(\lambda_{21}, \lambda_{22},\lambda_{23})=(1,1,1);
$$
$$
(\lambda_{78}, \lambda_{79},\lambda_{80})=(\lambda_{24}, \lambda_{25},\lambda_{26})=(0,1,0).
$$

Taking into account the periodicity of the mask, we obtain a tree

\unitlength=0.60mm
  \begin{picture}(240,120)
   \small
       \put(202,48){\line (-1,1){16}}
     \put(202,48){\line (0,1){16}}
     \put(202,48){\line (1,1){16}}

     \put(185,66){$\lambda_{24}$}
     \put(185,72){$0$}
     \put(198,66){$\lambda_{25}=1$}
      \put(221,66){$\lambda_{26}=0$}
      \put(197,66){\textcolor{green}{$\bigcirc $}}

      \put(200,72){\line (-1,1){12}}
     \put(200,72){\line (0,1){12}}
     \put(200,72){\line (1,1){12}}

     \put(200,92){\line (-1,1){16}}
     \put(200,92){\line (0,1){16}}
     \put(200,92){\line (1,1){16}}

      \put(120,86){$\lambda_{17}$}
      \put(133,86){$\lambda_{18}$}
     \put(142,86){$\lambda_{19}$}
      \put(151,86){$\lambda_{20}$}
     \put(170,86){$1=\lambda_{75}=$}
     \put(198,86){$\lambda_{76}=$}
      \put(213,86){$\lambda_{77}$}
      \put(223,86){$\lambda_{78}$}
            \put(235,86){$\lambda_{79}$}

      \put(213,86){\textcolor{green}{$\bigcirc $}}
      \put(197,86){\textcolor{green}{$\bigcirc $}}
      \put(180,86){\textcolor{green}{$\bigcirc $}}

      \put(176,90){\line (0,1){16}}
      \put(176,90){\line (-1,1){16}}
      \put(176,90){\line (-2,1){32}}

        \put(220,90){\line (1,1){16}}
     \put(220,90){\line (0,1){16}}
     \put(220,90){\line (2,1){32}}

      \put(180,108){$\lambda_{228}$}
     \put(194,108){$\lambda_{229}$}
      \put(208,108){$\lambda_{230}$}

       \put(138,108){$\lambda_{225}$}
     \put(154,108){$\lambda_{226}$}
      \put(168,108){$\lambda_{227}$}

      \put(220,108){$\lambda_{231}$}
     \put(237,108){$\lambda_{232}$}
      \put(254,108){$\lambda_{233}$}

      \put(225,72){\line (0,1){12}}
      \put(225,72){\line (1,1){12}}
      \put(225,72){\line (2,1){24}}

      \put(245,86){$\lambda_{80}=0$}

     \put(43,48){\line (-1,2){8}}
     \put(45,48){\line (0,1){16}}
     \put(46,48){\line (1,2){8}}

     \put(30,66){$\lambda_{9}$}
      \put(48,66){$\lambda_{11}$}
       \put(39,66){$\lambda_{10}$}

       \put(32,72){0}
      \put(51,72){0}
       \put(42,72){0}

       \put(60,66){$\lambda_{12}$}
       \put(71,66){$\lambda_{13}$}
       \put(81,66){$\lambda_{14}$}

        \put(111,48){\line (1,2){8}}
     \put(111,48){\line (0,1){16}}
     \put(111,48){\line (-1,1){16}}

      \put(30,86){$\lambda_{0}$}
      \put(39,86){$\lambda_{1}$}
       \put(48,86){$\lambda_{2}$}

            \put(60,86){$ \ . \ .\  .\ $}
       \put(75,86){$\lambda_{26}$}
       \put(87,86){$\lambda_{0}$}
       \put(96,86){$\lambda_{1 \ .\ .\  .}$}
       \put(105,86){$\lambda_{2 \ .\ .\  .}$}

      \put(93,66){$\lambda_{15},$}
      \put(103,66){$\lambda_{16},$}
       \put(113,66){$\lambda_{17}$}

        \put(124,66){$\lambda_{18}$}
      \put(134,66){$\lambda_{19}$}
       \put(144,66){$\lambda_{20}$}
        \put(124,72){0}
      \put(134,72){0}
       \put(144,72){0}

       \put(94,72){0}
      \put(104,72){0}
       \put(114,72){0}

       \put(175,72){1}
       \put(165,72){1}
       \put(155,72){1}

       \put(62,72){0}
       \put(72,72){0}
       \put(82,72){0}

   \put(168,25){\line (0,1){16}}
     \put(168,25){\line (-1,1){16}}
     \put(168,25){\line (2,1){32}}

     \put(143,45){$\lambda_{6}$}
     \put(168,43){$\lambda_{7}=0$}
      \put(198,43){$\lambda_{8}=1$}
       \put(198,43){\textcolor{green}{$\bigcirc $}}

       \put(155,66){$\lambda_{21}$}
     \put(165,66){$\lambda_{22}$}
      \put(173,66){$\lambda_{23}$}

      \put(168,45){\line (2,3){12}}
     \put(168,45){\line (0,1){18}}
     \put(168,45){\line (-2,3){12}}

\put(143,48){\line (-1,1){15}}
     \put(143,48){\line (-1,3){5}}
     \put(143,48){\line (1,2){8}}

   \put(70,25){\line (-3,2){24}}
     \put(70,25){\line (0,1){16}}
     \put(70,25){\line (2,1){38}}

     \put(43,43){$\lambda_{3}$}
      \put(110,43){$\lambda_{5}$}
      \put(69,43){$\lambda_{4}$}

     \put(165,19){$\lambda_{2}=1$}
     \put(165,19){\textcolor{green}{$\bigcirc $}}
     \put(68,19){$\lambda_{1}$}

        \put(73,48){\line (1,2){8}}
     \put(73,48){\line (0,1){16}}
     \put(73,48){\line (-1,2){8}}

 \put(130,-2){$\lambda_{0}=1$}
  \put(130,0){\line (-3,1){57}}
  \put(149,3){\line (1,1){16}}
    \put(-5,108){$G_3^\bot \setminus G_{2}^\bot $:}
  \put(-5,87){$G_2^\bot \setminus G_{1}^\bot $:}
  \put(-5,66){$G_1^\bot \setminus G_{0}^\bot $:}
  \put(-5,43){$G_{0}^\bot \setminus G_{-1}^\bot $:}
  \put(-5,19){$G_{-1}^\bot \setminus G_{-2}^\bot $:}
  \put(-5,-2){$G_{-2}^\bot  $:}
    \end{picture}\\

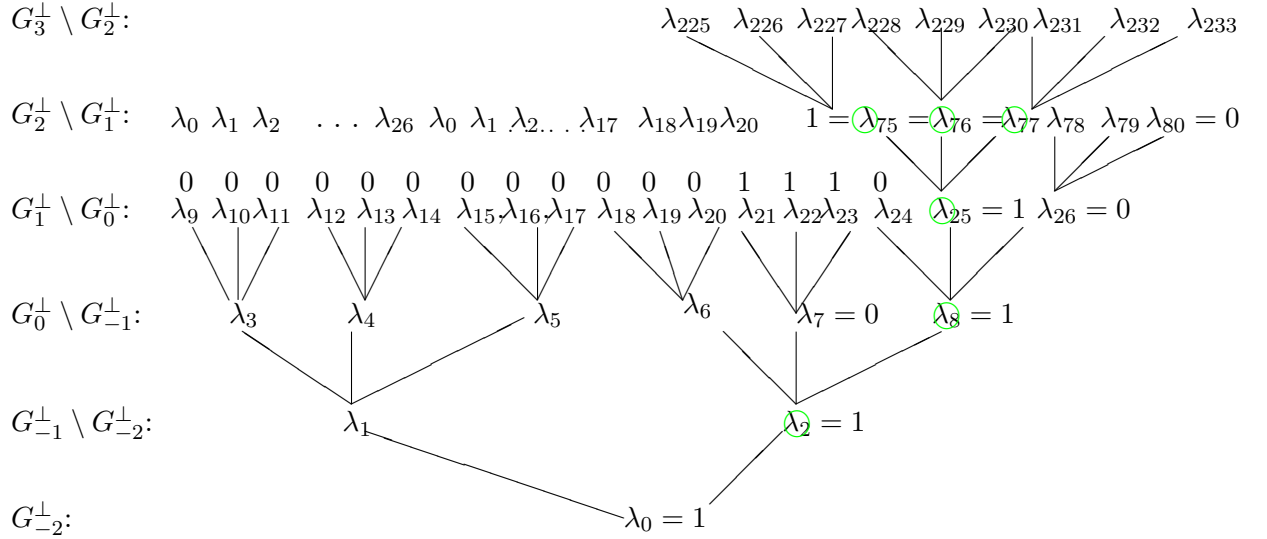
    \centerline{Figure 5}

Using the resulting tree with three paths, we find the Fourier transform of the scaling function

$$\hat\varphi_(\chi)=\left\{
\begin{array}{ll}
1,& \chi\in G^\bot_{-2} ,\\
\alpha_0,& \chi\in G^\bot_{-2}r_{-2}^0r_{-1}^1r_{0}^2r_1^2,\\
\alpha_1,& \chi\in G^\bot_{-2}r_{-2}^1r_{-1}^1r_{0}^2r_1^2,\\
\alpha_2,& \chi\in G^\bot_{-2}r_{-2}^2r_{-1}^1r_{0}^2r_1^2.\\
\end{array} \right.$$
We can choose $\alpha_0=\alpha_1=\alpha_2=1.$
Using the equality,
$$\int_{G_n^\bot}(\chi,x)d\nu (\chi) =p^n {\bf 1}_{G_n}(x)$$
we reconstruct the scaling function
$$
\varphi (x)=\int_G \hat\varphi_(\chi)(\chi,x)d\nu (\chi)=$$

$$=\int_{G^\bot_{-2}}(\chi,x)d\nu (\chi)
+\int_{G^\bot_{-2}r_{-2}^0 r_{-1}^1r_{0}^2r_{1}^2}(\chi,x)d\nu (\chi)+
$$
$$
\int_{G^\bot_{-2}r_{-2}^1 r_{-1}^1r_{0}^2r_{1}^2}(\chi,x)d\nu (\chi)+
\int_{G^\bot_{-2}r_{-2}^2 r_{-1}^1r_{0}^2r_{1}^2}(\chi,x)d\nu (\chi)=
$$
$$
=\frac{1}{p^2}{\bf 1}_{G_{-2}}(x)+\frac{1}{p^2}{\bf 1}_{G_{-2}}(x)(r_{-1}(x)r_0^2(x)r_1^2(x) (1+r_{-2}(x)+r_{-2}^2(x)))
$$
\subsection{Example 3. Path from node 130}
Let us form other pairwise distinct vectors (see Theorem 3.1)
  $$ (\alpha_{-1},\alpha_{0})=(1,2),\ (\alpha_{0},\alpha_{1})=(2,1),\ (\alpha_{1},\alpha_{2})=(1,1).
  $$
  Thus, we obtain the coefficients $\alpha_{-1}=1,\alpha_{0}=2,\alpha_{1}=1,\alpha_{2}=1$.
  For  $\alpha_{-2}=1$ we obtain  $n=1+3\cdot 1+3^2\cdot 2+3^3\cdot 1+3^4\cdot 1 =130$.
  We select a path from node 130
    and find a path $$((((130 \ {\rm div} \ 3) \ {\rm div}\  3) \ {\rm div} \ 3) \ {\rm div}\  3)=130 \rightarrow 43 \rightarrow 14 \rightarrow 4\rightarrow 1$$
     that generates a scaling function (see Fig. 6).
   This path is highlighted by circles. By construction, on this path we have
  $$\lambda_{129}=\lambda_{130}=\lambda_{131}=0, \lambda_{43}=\lambda_{14}=\lambda_{4}=\lambda_{1}=\lambda_{0}=1.$$

\unitlength=0.60mm
  \begin{picture}(240,120)
   \small

       \put(108,90){\line (-1,1){16}}
      \put(108,90){\line (0,1){16}}
      \put(108,90){\line (1,1){16}}

      \put(88,108){$\lambda_{129}$}
      \put(103,108){$\lambda_{130}$}
      \put(118,108){$\lambda_{131}$}

       \put(87,66){$\lambda_{12}$}
       \put(98,66){$\lambda_{13}$}
       \put(109,66){$\lambda_{14}$}

       \put(111,70){\line (0,1){16}}
        \put(111,70){\line (-1,1){16}}
         \put(111,70){\line (1,1){16}}
         \put(110,66){\textcolor{green}{$\bigcirc $}}

             \put(60,86){$ \ . \ .\  .\ $}
       \put(107,86){$\lambda_{43}$}
       \put(122,86){$\lambda_{44}=0$}
        \put(107,86){\textcolor{green}{$\bigcirc $}}
       \put(80,86){$0=\lambda_{42 \ .\ .\  .}$}

              \put(88,72){0}
       \put(98,72){0}

   \put(100,25){\line (-1,1){16}}
     \put(100,25){\line (0,1){27}}
     \put(100,25){\line (1,1){16}}

     \put(73,43){$\lambda_{3}=0$}
      \put(112,43){$\lambda_{5}=0$}
      \put(98,43){$\lambda_{4}$}
      \put(98,43){\textcolor{green}{$\bigcirc $}}

     \put(125,19){$\lambda_{2}=0$}
     \put(98,19){\textcolor{green}{$\bigcirc $}}
     \put(98,-2){\textcolor{green}{$\bigcirc $}}
     \put(98,19){$\lambda_{1}=1$}

        \put(100,48){\line (1,2){8}}
     \put(100,48){\line (0,1){16}}
     \put(100,48){\line (-1,2){8}}

 \put(98,-2){$\lambda_{0}=1$}

  \put(100,3){\line (0,1){15}}
  \put(100,3){\line (2,1){30}}
    \put(-5,108){$G_3^\bot \setminus G_{2}^\bot $:}
  \put(-5,87){$G_2^\bot \setminus G_{1}^\bot $:}
  \put(-5,66){$G_1^\bot \setminus G_{0}^\bot $:}
  \put(-5,43){$G_{0}^\bot \setminus G_{-1}^\bot $:}
  \put(-5,19){$G_{-1}^\bot \setminus G_{-2}^\bot $:}
  \put(-5,-2){$G_{-2}^\bot  $:}
    \end{picture}\\

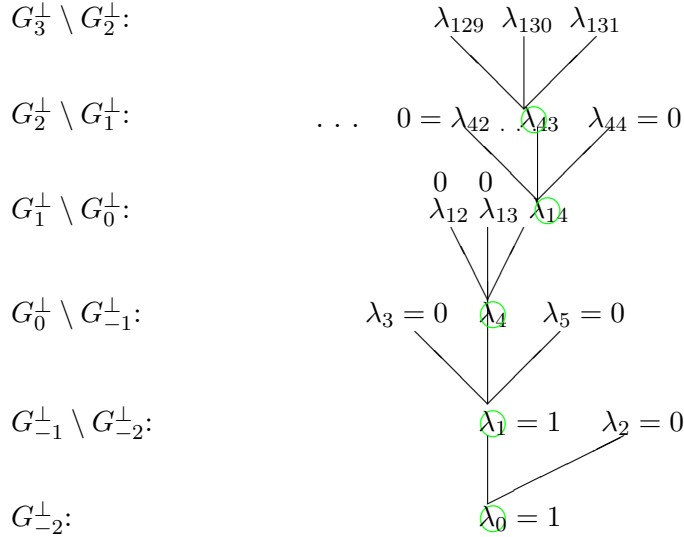
\centerline{Figure 6}

Taking into account the periodicity of the mask, we obtain a tree

\unitlength=0.60mm
  \begin{picture}(240,120)
   \small
     \put(202,48){\line (-1,2){8}}
     \put(202,48){\line (0,1){16}}
     \put(202,48){\line (1,1){16}}

     \put(230,66){$\lambda_{24}$}
     \put(242,66){$\lambda_{25}$}
      \put(256,66){$\lambda_{26}$}
       \put(230,71){0}
     \put(242,71){0}
      \put(256,71){0}

      \put(133,86){$\lambda_{45}$}
     \put(142,86){$\lambda_{46}$}
      \put(151,86){$\lambda_{47...}$}

      \put(165,86){$\lambda_{51}$}
     \put(175,86){$\lambda_{52}$}
     \put(185,86){$\lambda_{53\ .}$}

      \put(200,86){$\lambda_{54}$}
      \put(210,86){$\lambda_{55\ .\ .\ .}$}
          \put(235,86){$\lambda_{78}$}
      \put(245,86){$\lambda_{79}$}
      \put(255,86){$\lambda_{80}$}

      \put(98,90){\line (-3,1){48}}
      \put(98,90){\line (-1,1){16}}
      \put(98,90){\line (-2,1){32}}

       \put(108,90){\line (-1,1){16}}
      \put(108,90){\line (0,1){16}}
      \put(108,90){\line (1,1){16}}
      \put(88,108){$\lambda_{129}$}
      \put(103,108){$\lambda_{130}$}
      \put(118,108){$\lambda_{131}$}

        \put(125,90){\line (2,1){32}}
     \put(125,90){\line (1,1){16}}
     \put(125,90){\line (3,1){48}}

       \put(43,108){$\lambda_{126}$}
     \put(58,108){$\lambda_{127}$}
      \put(75,108){$\lambda_{128}$}

      \put(136,108){$\lambda_{132}$}
     \put(152,108){$\lambda_{133}$}
      \put(170,108){$\lambda_{134}$}

     \put(43,48){\line (-1,2){8}}
     \put(45,48){\line (0,1){16}}
     \put(46,48){\line (1,2){8}}

     \put(30,66){$\lambda_{9}$}
      \put(48,66){$\lambda_{11}$}
       \put(39,66){$\lambda_{10}$}

       \put(32,72){0}
      \put(51,72){0}
       \put(42,72){0}

       \put(87,66){$\lambda_{12}$}
       \put(98,66){$\lambda_{13}$}
       \put(109,66){$\lambda_{14}$}

       \put(111,70){\line (0,1){16}}
        \put(111,70){\line (-1,2){8}}
         \put(111,70){\line (1,2){8}}
         \put(110,66){\textcolor{green}{$\bigcirc $}}

        \put(141,48){\line (1,2){8}}
     \put(141,48){\line (0,1){16}}
     \put(141,48){\line (-1,1){16}}

      \put(30,86){$\lambda_{27}$}
      \put(39,86){$\lambda_{28}$}
       \put(48,86){$\lambda_{29}$}

            \put(58,86){$. .. \ $}
       \put(107,86){$\lambda_{43}$}
       \put(120,86){$\lambda_{44}$}
        \put(107,86){\textcolor{green}{$\bigcirc $}}
       \put(96,86){$\lambda_{42 }$}

 \put(65,86){$\lambda_{39}$}
       \put(75,86){$\lambda_{40}$}
       \put(85,86){$\lambda_{41}$}

      \put(120,66){$\lambda_{15},$}
      \put(132,66){$\lambda_{16},$}
       \put(144,66){$\lambda_{17}$}

      \put(120,72){1}
      \put(132,72){1}
       \put(144,72){1}

        \put(154,66){$\lambda_{18}$}
      \put(164,66){$\lambda_{19}$}
       \put(174,66){$\lambda_{20}$}
        \put(154,72){0}
      \put(164,72){0}
       \put(174,72){0}

       \put(190,72){0}
       \put(202,72){0}
       \put(215,72){0}

             \put(88,72){0}
       \put(98,72){0}
         \put(168,25){\line (0,1){16}}
     \put(168,25){\line (4,1){62}}
     \put(168,25){\line (2,1){32}}

     \put(228,45){$\lambda_{8}$}
     \put(168,43){$\lambda_{6}$}
      \put(198,43){$\lambda_{7}$}

       \put(189,66){$\lambda_{21}$}
     \put(200,66){$\lambda_{22}$}
      \put(215,66){$\lambda_{23}$}

      \put(168,45){\line (2,3){12}}
     \put(168,45){\line (0,1){18}}
     \put(168,45){\line (-2,3){12}}

     \put(233,48){\line (2,1){30}}
     \put(233,48){\line (0,1){20}}
     \put(233,48){\line (1,1){18}}

   \put(70,25){\line (-3,2){24}}
     \put(70,25){\line (3,2){27}}
     \put(70,25){\line (4,1){68}}

     \put(43,43){$\lambda_{3}=0$}
      \put(138,43){$\lambda_{5}=0$}
      \put(98,43){$\lambda_{4}=1$}
      \put(98,43){\textcolor{green}{$\bigcirc $}}

     \put(165,19){$\lambda_{2}=0$}
     \put(68,19){\textcolor{green}{$\bigcirc $}}
     \put(68,19){$\lambda_{1}=1$}

        \put(100,48){\line (1,2){8}}
     \put(100,48){\line (0,1){16}}
     \put(100,48){\line (-1,2){8}}

 \put(130,-2){$\lambda_{0}=1$}
  \put(130,0){\line (-3,1){57}}
  \put(149,3){\line (1,1){16}}
    \put(-5,108){$G_3^\bot \setminus G_{2}^\bot $:}
  \put(-5,87){$G_2^\bot \setminus G_{1}^\bot $:}
  \put(-5,66){$G_1^\bot \setminus G_{0}^\bot $:}
  \put(-5,43){$G_{0}^\bot \setminus G_{-1}^\bot $:}
  \put(-5,19){$G_{-1}^\bot \setminus G_{-2}^\bot $:}
  \put(-5,-2){$G_{-2}^\bot  $:}
    \end{picture}\\

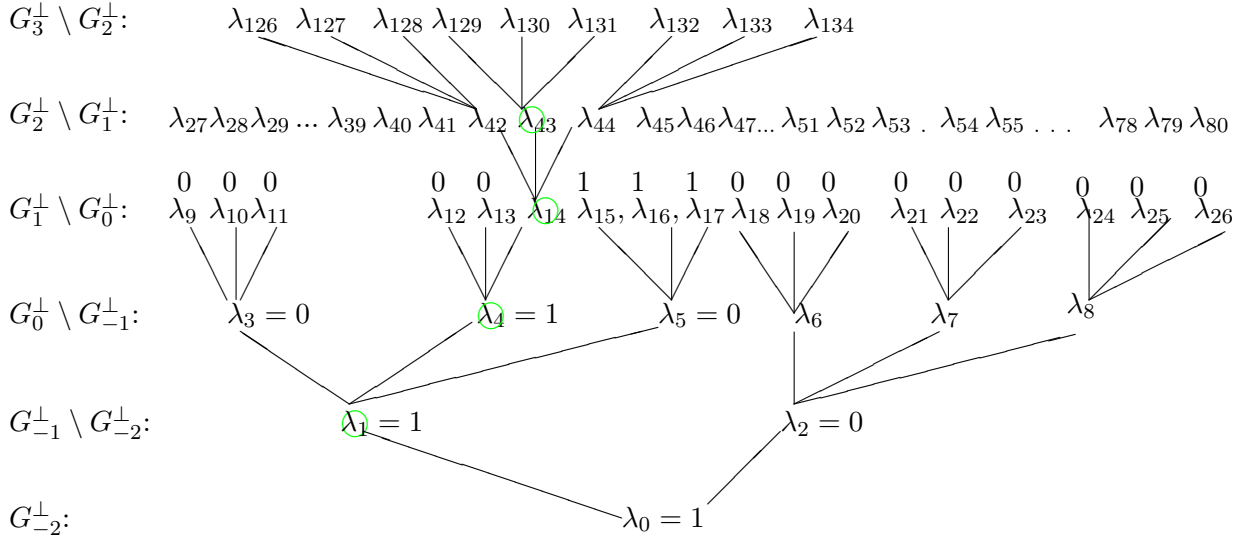
\centerline{Figure 7}

    It is clear that
    $$
    (\lambda_{27},\lambda_{28},...,\lambda_{53})= (\lambda_{54},\lambda_{55},...,\lambda_{80})=(\lambda_{0},\lambda_{1},...,\lambda_{26})
    $$
    and $m_0(\xi)m_0(\xi\mathcal{A}^{-1})m_0(\xi\mathcal{A}^{-2})=0$ on $G_3^\bot\setminus G_2^\bot$.
Using the resulting tree, we find the Fourier transform of the scaling function
$$\hat\varphi_(\chi)= {\bf 1}_{G_{-2}^\bot} (\xi) + {\bf 1}_{G_{-2}^\bot r_{-2}} (\xi) +
 {\bf 1}_{G_{-2}^\bot r_{-2}r_{-1}} (\xi)+
 {\bf 1}_{G_{-2}^\bot r_{-2}^2r_{-1}r_0} (\xi)+
 {\bf 1}_{G_{-1}^\bot r_{-1}^2r_0} (\xi).
$$

Using the equality,
$$\int_{G_n^\bot r_n^{\alpha_n} r_{n+1}^{\alpha_{n+1}}  ... r_{n+s}^{\alpha_{n+s}}}(\chi,x)d\nu (\chi) =p^n r_n^{\alpha_n} r_{n+1}^{\alpha_{n+1}}  ... r_{n+s}^{\alpha_{n+s}}{\bf 1}_{G_n}(x)$$
we reconstruct the scaling function.

$$
\varphi (x)=\int_G \hat\varphi_(\chi)(\chi,x)d\nu (\chi)=$$

$$=\int_{G^\bot_{-2}}(\chi,x)d\nu (\chi)
+\int_{G^\bot_{-2}r_{-2}}(\chi,x)d\nu (\chi)+
$$
$$
\int_{G^\bot_{-2}r_{-2} r_{-1}}(\chi,x)d\nu (\chi)+
\int_{G^\bot_{-2}r_{-2}^2 r_{-1}r_{0}}(\chi,x)d\nu (\chi)+
\int_{G^\bot_{-1}r_{-1}^2 r_{0}}(\chi,x)d\nu (\chi)=
$$
$$
=\frac{1}{3^2}{\bf 1}_{G_{-2}}(x)(1+ r_{-2}(x)
+r_{-2}(x)r_{-1}(x)+
r_{-2}^2(x)r_{-1}(x)r_0(x))+
$$
$$
+\frac{1}{3^2}r_{-1}^2(x)r_0(x){\bf 1}_{G_{-1}}(x).
$$

\end{document}